\documentclass[
final,nomarks
]{dmtcs-episciences}

\usepackage[utf8]{inputenc}
\usepackage{subfigure}

\usepackage[round]{natbib}

\author[Mikl\'os B\'ona]{Mikl\'os B\'ona\affiliationmark{1}\thanks{Partially supported by  Simons Foundation Collaboration Grant 940024}  }
\title[Stack-sorting preimages and 0-1 trees]{Stack-sorting preimages and 0-1 trees}
\affiliation{
  Uninversity of Florida}

\newtheorem{theorem}{Theorem}[section]
\newtheorem{proposition}[theorem]{Proposition}

\newtheorem{example}[theorem]{Example}

\newcommand{\av}{\hbox{Av}}
\newtheorem{definition}[theorem]{Definition}

\begin{document}

\publicationdata{vol. 28:1, Permutation Patterns 2025}{2026}{1}{10.46298/dmtcs.16393}{2025-08-22; 2025-08-22; 2026-01-23; 2026-02-27}{2026-02-28}
\maketitle
\begin{abstract}
We define a class of partially labeled trees and use them to find simple proofs for two recent enumeration results of Colin Defant concerning stack-sorting preimages of permutation classes.
\end{abstract}

\section{Introduction}
{\em Stack sorting} or {\em West stack sorting} of permutations is an interesting and well-studied field of research. See Chapter 8 of \cite{b22} for an introduction. 
The stack sorting operation $s$ has three natural,  equivalent definitions. Here we mention one of them. We define $s$ on all finite permutations $p$ {\em and their substrings} by
setting $s(\varepsilon)=\varepsilon$ for the empty permutation, and setting $s(p)=s(LnR)=s(L)s(R)n$, where $n$ is the maximum entry of the string $p$, while $L$ is the substring of entries
on the left of $n$, and $R$ is the substring on the entries on the right of $n$. Note that $L$ and $R$ can be empty. 

The questions, methods and results about the stack sorting algorithm are intrinsically connected to the theory of pattern-avoiding permutations, a large and rapidly evolving field. 
 We say that a permutation $p$ {\em contains} the pattern (or subsequence) $q=q_1q_2\cdots q_k$ 
if there is a $k$-element set of indices $i_1<i_2< \cdots <i_k$ such that $p_{i_r} < p_{i_s} $ if and only
if $q_r<q_s$.  If $p$ does not contain $q$, then we say that $p$ {\em avoids} $q$. For example, $p=3752416$ contains
$q=2413$, as the first, second, fourth, and seventh entries of $p$ form the subsequence 3726, which is order-isomorphic
to $q=2413$.  A recent survey on permutation 
patterns by Vatter can be found in \cite{v15}.

Recently, in a series of papers \cite{aust,lothar,enum} Colin Defant studied the size of $s^{-1}(Class_n)$, that is, he studied the number of permutations $p$ of length $n$
for which $s(p)$ belongs to a certain permutation class.  Let $\av_n(q)$ denote the set of permutations of length $n$ that avoid the pattern $q$. The classic result that the number
of permutations $p$ of length $n$ for which $s(p)$ is the identity permutation is the Catalan number $C_n$  is a special case of this line of work.  Namely, it is the result that
$|s^{-1}(\av_n(21))|=C_n$. Considering various permutation classes defined by sets of patterns of length three,  Defant has proved the chain
of equalities

\begin{eqnarray*}\sum_{n\geq 1} |s^{-1}(\av_n(132,312))| z^n & = &  \sum_{n\geq 1} |s^{-1}(\av_n(231,312))| z^n \\ = \sum_{n\geq 1} |s^{-1}(\av_n(132,231))|z^n 
& = & \frac{1-2z-\sqrt{1-4z-4z^2}}{4z} .\end{eqnarray*}

This was not an easy feat. The enumeration formula for the first set of permutations was proved in \cite{lothar}, the formula for the second set was proved in \cite{aust}, and the formula for the third one was proved in \cite{enum}, and the techniques used in these proofs were also different. 

In this paper, we give simple, but perhaps interesting proofs for the above enumeration result for $|s^{-1}(\av_n(132,312))| $ and $ |s^{-1}(\av_n(231,312))|$.
A key element in a proof will be the notion of a certain kind of partially labeled trees that we call $(0,1)$-trees. We will show that the numbers counting our permutations 
satisfy the same recurrence relation as these trees.

\section{0-1-Trees}
A {\em plane binary tree} is a rooted plane tree on unlabeled vertices in which every non-leaf vertex has one or two children, and every child is a left child or a right child, even if it is the
only child of its parent. It is well-known that the number of such trees on $n$ vertices is the $n$th Catalan number $C_n={2n\choose n}/(n+1)$. We will consider an enhanced version
of these trees. 

\begin{definition} A {\em 0-1-tree} is a binary plane tree in which every vertex that has two children is labeled 0 or 1. \end{definition}

Figure 1 shows the six 0-1-trees on three vertices. 
\begin{figure}[htb] \label{trees-0-1}
\begin{center}\includegraphics[width=6cm]{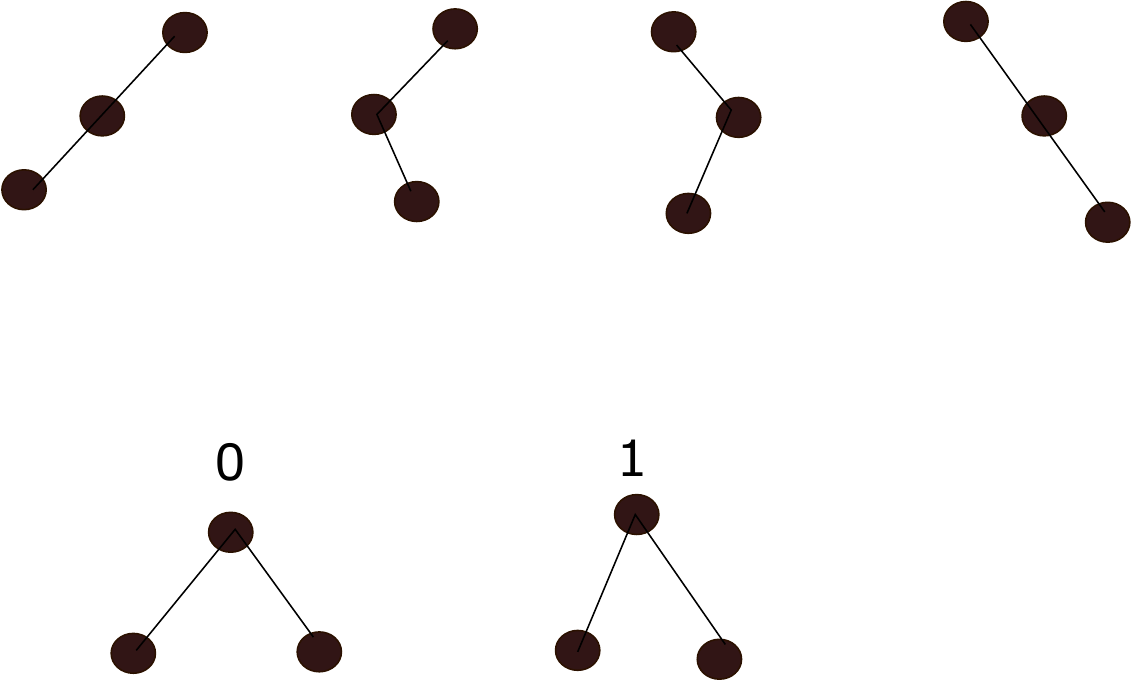}
\caption{The six 0-1 trees on three vertices.}
\end{center} \end{figure}

Let $a_n$ be the number of 0-1-trees on $n$ vertices, and let us set $a_0=0$. 
\begin{proposition}
Let $A(z)=\sum_{n\geq 0}a_nz^n$ be the ordinary generating function of the sequence $\{a_n\}_{n\geq 0}$. Then
\[A(z)=\frac{1-2z-\sqrt{1-4z-4z^2}}{4z} .\]
\end{proposition}

\begin{proof}
The root of a 0-1-tree that consists of more than one vertex either has a left child only, or a right child only, or two children. It is only in the last case that the root will get a label. This implies
 the recurrence relation
\begin{equation} \label{recurrence} a_n=2a_{n-1}+2\sum_{i=2}^{n-1} a_{i-1}a_{n-i}, \end{equation} for $n\geq 2$ with $a_1=1$, 
or, equivalently, the functional equation
\[A(z)=z+2zA(z)+2zA^2(z),\] with $A(0)=0$. Solving this quadratic equation and choosing the solution that satisfies the initial condition yields our result.
\end{proof}

The sequence $\{a_n\}_{n\geq 1}$ appears in \cite{sloane} as sequence A071356, with the indices shifted by 1. In \cite{hossain}, Chetak Hossain suggests the name {\em Boolean-Catalan
numbers} for the numbers $a_n$. Our notion of (0-1)-trees supports this suggestion. Indeed, we take a structure counted by the Catalan numbers, namely the binary trees, and 
enhance it by a "Boolean" structure, namely zeros and ones.

\section{Determining $|s^{-1}(\av_n(132,312))|$.}

Let us set $b_n=|s^{-1}(\av_n(132,312))|$ for ease of notation. Set $b_0=0$.
\begin{theorem} The sequence $b_n$ satisfies the recurrence relation (\ref{recurrence}) that is satisfied by the sequence $a_n$. As $a_1=b_1=1$, this implies that $a_n=b_n$ for all $n$.
\end{theorem}

For the rest of this paper, for a permutation $p$ of length $n$, we use the notation $p=LnR$, where $L$ and $R$ denote the possibly empty string of entries on the left of the entry $n$
and on the right of the entry $n$. We recall the crucial identity $s(p)=s(L)s(R)n$. 

A {\em left-to-right maximum} (resp. minimum) in a permutation is an entry that is larger (resp. smaller) than everything on its left. 
For shortness, we will call left-to-right maxima simply {\em maxima}, and left-to-right minima simply {\em minima}. 

\begin{proof}
Intuitively speaking, we will show that there are two ways to combine two shorter permutations of the desired kind into one, just like there are two ways to join two smaller
(0-1)-trees to a common root. In order to make this idea more precise, we have to analyze the structure of our permutations. 

 If a permutation avoids 132 and 312, then it ends with $1$ or $n$, and in general, moving right to left,
 each entry is a left-to-right minimum or left-to-right maximum, and only one of those, {\em except} the leftmost entry, which is both. (This is also an easy way to see that
 $|\av_n(132,312))| =2^{n-1}$.)  See Figure 2 for an illustration.

\begin{figure}[htb] \label{132-312-perm}
\begin{center}\includegraphics[width=6cm]{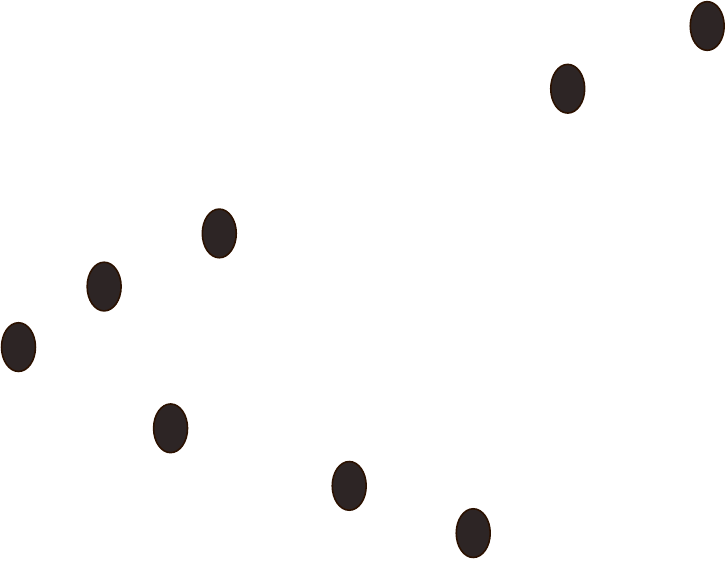}
\caption{A generic permutation avoiding 132 and 312.}
\end{center} \end{figure}

 Now let $p=LnR$. It follows from the description in the previous paragraph that then $s(p)=s(L)s(R)n$ avoids both 132 and 312 if and only if $s(L)$ and $s(R)$ avoid these two patterns, the  maxima of $s(L)$ are smaller than the maxima of $s(R)$ and the minima of $s(L)$ are larger than the minima of $s(R)$. In other words, {\em all} the entries in $L$ are smaller than
 the maxima in $s(R)$ and
larger than the minima in $s(R)$. (Note that the set of entries in $L$ is the same as the set of entries in $s(L)$.) This means that for each ordered pair  $(\ell,r)$ of nonempty patterns for which $s(\ell)$ and $s(r)$ avoid 132 and 312 and $|\ell|+|r|=n-1$, there are exactly
{\em two}  permutations $p\in s^{-1}(\av_n(132,312))$ so that $\ell$ is isomorphic to the string $L$ of entries preceding $n$ in $p$ and $r$ is isomorphic to the string $R$ of entries following $n$ in $p$. This is  because the leftmost entry of $s(R)$ can be viewed as a left-to-right minimum or as a left-to-right maximum in $s(R)$, hence its value in $p$ can be chosen to be smaller than
all entries in $L$, or larger than all entries in $L$. See Example \ref{twochoices} and Figure \ref{twocases} for illustrations. Once that choice is made, $p=LnR$ is determined, since $L$ is isomorphic to $\ell$, and $R$ is isomorphic to $r$.
This argument is valid as long as $L$ and $R$ are not empty, so when $n$ is in position $i\in [2,n-1]$. This case contributes $2\sum_{i=2}^{n-1} b_{i-1}b_{n-i}$ permutations to 
$b_n$.

Finally, it is clear that if $p'\in s^{-1}(\av_{n-1}(132,312))$, then prepending or postpending $p'$ with $n$, we get a permutation $p\in  s^{-1}(\av_{n-1}(132,312))$. This case contributes $2b_{n-1}$ permutations to $b_n$, completing the proof. 
\end{proof}

\begin{example} \label{twochoices}
Let $s(\ell) =2134$ and let $s(r)=32145$.  Then the entries 4 and 5 are maxima of $s(r)$, the entries  2 and 1 are minima of $s(r)$, and the entry 3 is both. Therefore, we have two possibilities.
\begin{enumerate}
\item We can treat 3 as a minimum, and obtain the strings $s(L)=5467$ and $s(R)=32189$, leading to the permutation $s(p)=5\ 4 \ 6 \ 7 \ 3\ 2\  1\  8 \ 9\ 10$, or 
\item we can treat 3 as a maximum, and obtain the strings $s(L)=4356$ and $s(R)=72189$, leading to the permutation  $s(p)=4\ 3 \ 5\ 6 \ 7\ 2 \ 1\  8\ 9  \ 10$.
\end{enumerate}
Note that in the first case, the leftmost entry of $s(R)$ is smaller than all entries in $s(L)$ (and so all entries in $L$); in the second case, the leftmost entry of $s(R)$ is larger than all entries in $L$. See Figure \ref{twocases} for an illustration.
\end{example}

\begin{figure}[htb] \label{twocases}
\begin{center}\includegraphics[width=7cm]{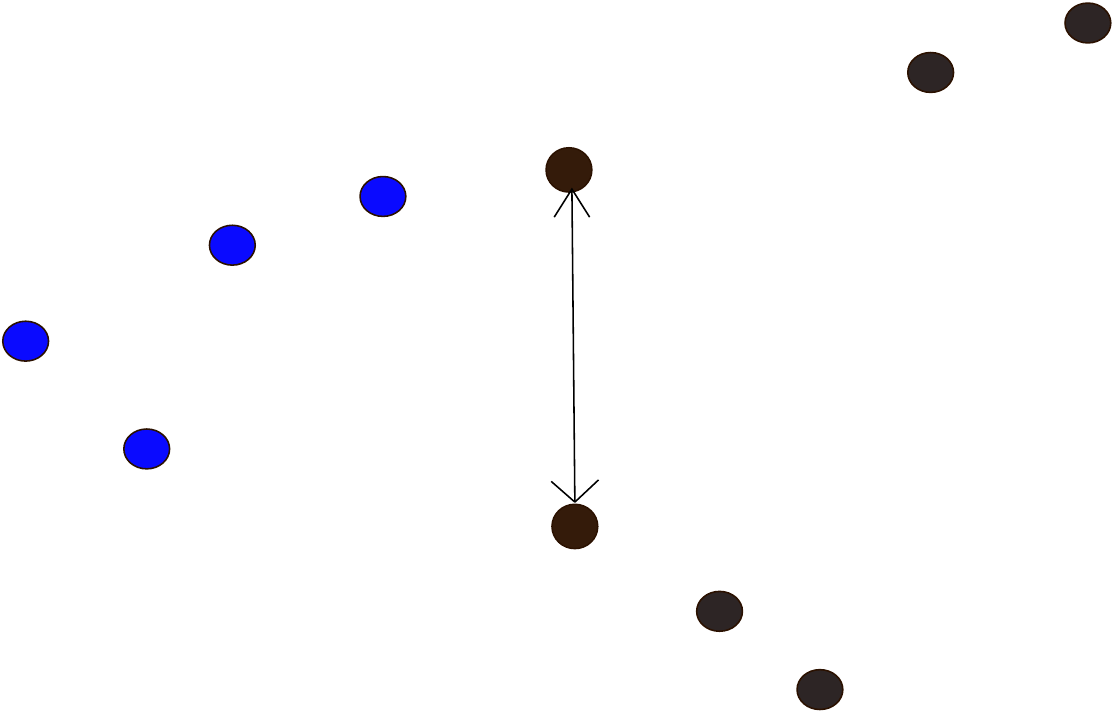}
\caption{The two possible values of the leftmost entry of $s(R)$ in $p$.}
\end{center} \end{figure}

\section{Determining $|s^{-1}(\av_n(231,312))|$.}

Let us set $f_n=|s^{-1}(\av_n(231,312))|$ for ease of notation. Set $f_0=0$.
\begin{theorem} The sequence $f_n$ satisfies the recurrence relation (\ref{recurrence}) that is satisfied by the sequence $a_n$. As $f_1=1$, this implies that $f_n=a_n$ for all $n$.
\end{theorem}

\begin{proof} Just as in the proof of the previous theorem, we will show that there are two ways to combine two shorter permutations with the desired properties into one. 
A permutation that avoids 231 and 312 is called a {\em layered} permutation. It consists of decreasing subsequences (the layers) so that the entries increase from each layer to the next. See Figure \ref{layered} for an illustration.

\begin{figure}[htb] \label{layered}
\begin{center}\includegraphics[width=6cm]{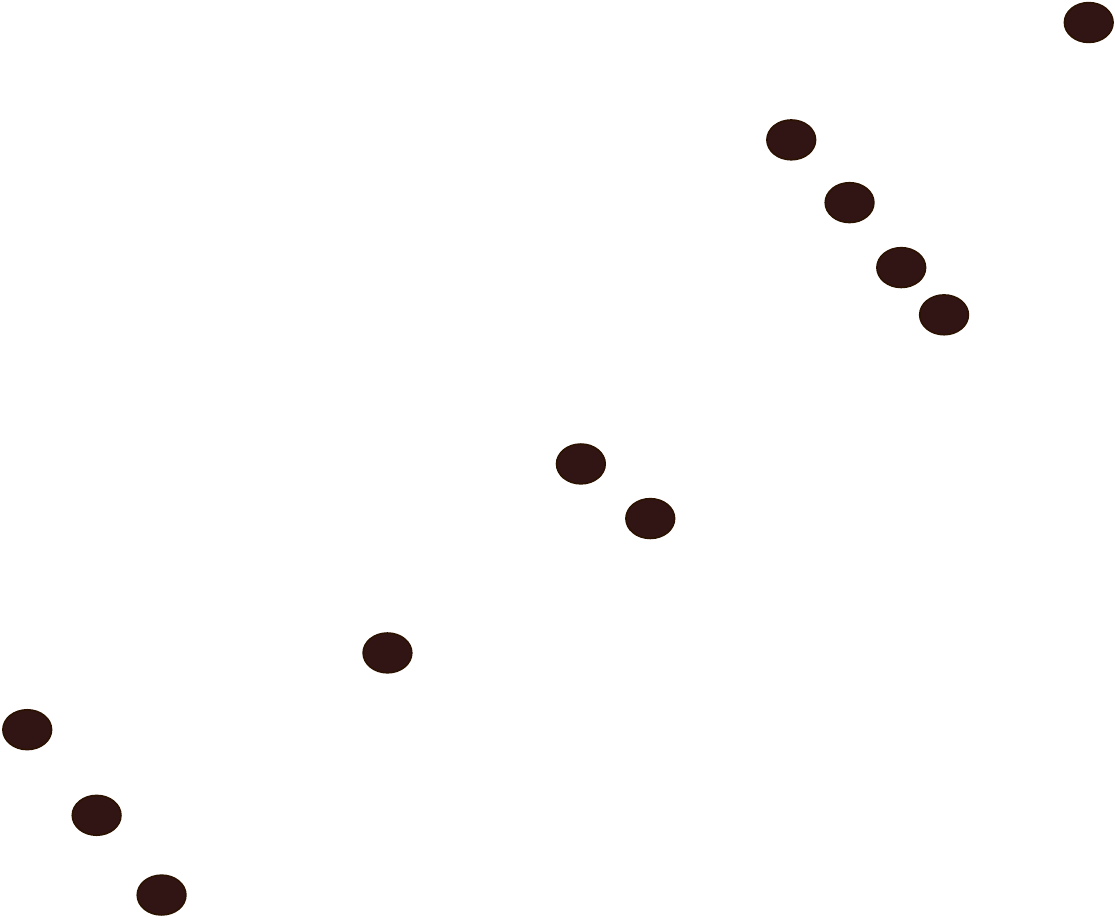}
\caption{A generic permutation avoiding 231 and 312, called a layered permutation.}
\end{center} \end{figure}

 Let $p=LnR$. It follows from the description in the previous paragraph that then $s(p)=s(L)s(R)n$ avoids both 231 and 312 if and only if $s(L)$ and $s(R)$ avoid these two patterns
and 
\begin{enumerate}
\item either all entries in $L$ are smaller than all entries in $R$, or
\item each entry of the last layer of $s(L)$ is larger than each entry of the first layer of $s(R)$, but otherwise each entry of $L$ is smaller than each entry of $R$.
\end{enumerate}

It is easy to see that in the first case, each layer of $s(L)$ and $s(R)$ will become a layer of $s(p)=s(L)s(R)n$, whereas in the second case, the last layer of $s(L)$ and the first layer
of $s(R)$ will unite to form  a new layer of $s(p)$. 
 
We claim that  for each ordered pair  $(\ell,r)$ of nonempty patterns for which $s(\ell)$ and $s(r)$ avoid 231 and 312 and $|\ell|+|r|=n-1$, there are exactly
{\em two}  permutations $p\in s^{-1}(\av_n(231,312))$ so that $\ell$ is isomorphic to the string $L$ of entries preceding $n$ in $p$ and $r$ is isomorphic to the string $R$ of entries following $n$ in $p$. This is because given $\ell$ and $r$ with the required properties, there are two possibilites for the underlying set of $\ell$. We can choose that set to be simply
the set $\{1,2,\cdots ,\ell\}$, and then all entries in $L$ will be smaller than all entries in $R$, and $s(L)s(R)n$ will be a layered permutation. This is the situation described in (1).
 Or, if the last layer of $\ell$ is of length
$\ell_1$ and the first layer of $r$ is of length $r_1$, then we can choose the underlying set of $\ell$ to be the set $\{1,2,\cdots ,\ell-\ell_1\} \cup \{\ell -\ell_1+r_1+1,\ell -\ell_1+r_1+2,
\cdots ,\ell+r_1\}$. In that case, we are in the situation described in (2).  Then the last layer of $s(L)$ and the first layer of $s(R)$ will unite to form a layer of length $\ell_1+r_1$ in $s(p)$.
 Once that choice is made, $p=LnR$ is determined, since $L$ is isomorphic to $\ell$, and $R$ is isomorphic to $r$.
 
Again, this argument is valid as long as $L$ and $R$ are not empty, so when $n$ is in position $i\in [2,n-1]$. This case contributes $2\sum_{i=2}^{n-1} f_{i-1}f_{n-i}$ permutations to 
$f_n$.

Finally, it is clear that if $p'\in s^{-1}(\av_{n-1}(231,312))$, then prepending or postpending $p'$ with $n$, we get a permutation $p\in  s^{-1}(\av_{n-1}(231,312))$. This case contributes $2f_{n-1}$ permutations to $f_n$, completing the proof. 
\end{proof}

\begin{example} \label{twonewchoices}
Let $s(\ell) =2134$ and let $s(r)=32145$.  Then there are two possibilities for $s(p)$. 
\begin{enumerate} \item If all entries in $L$ are smaller than all entries in $R$, then we have $s(p)=2 \ 1 \ 3 \ 4 \ 7 \ 6 \ 5 \ 8 \ 9 \ 10$, and
\item if the last layer of $s(L)$ consists of entries larger than the first layer of $s(R)$, then we have $s(p)=2 \ 1 \ 3 \ 7 \ 6 \ 5 \ 4 \ 8 \ 9 \ 10$.
\end{enumerate}
 See Figure 5 for an illustration.
\end{example}

\begin{figure}[htb] \label{twonewcases}
\begin{center}\includegraphics[width=7cm]{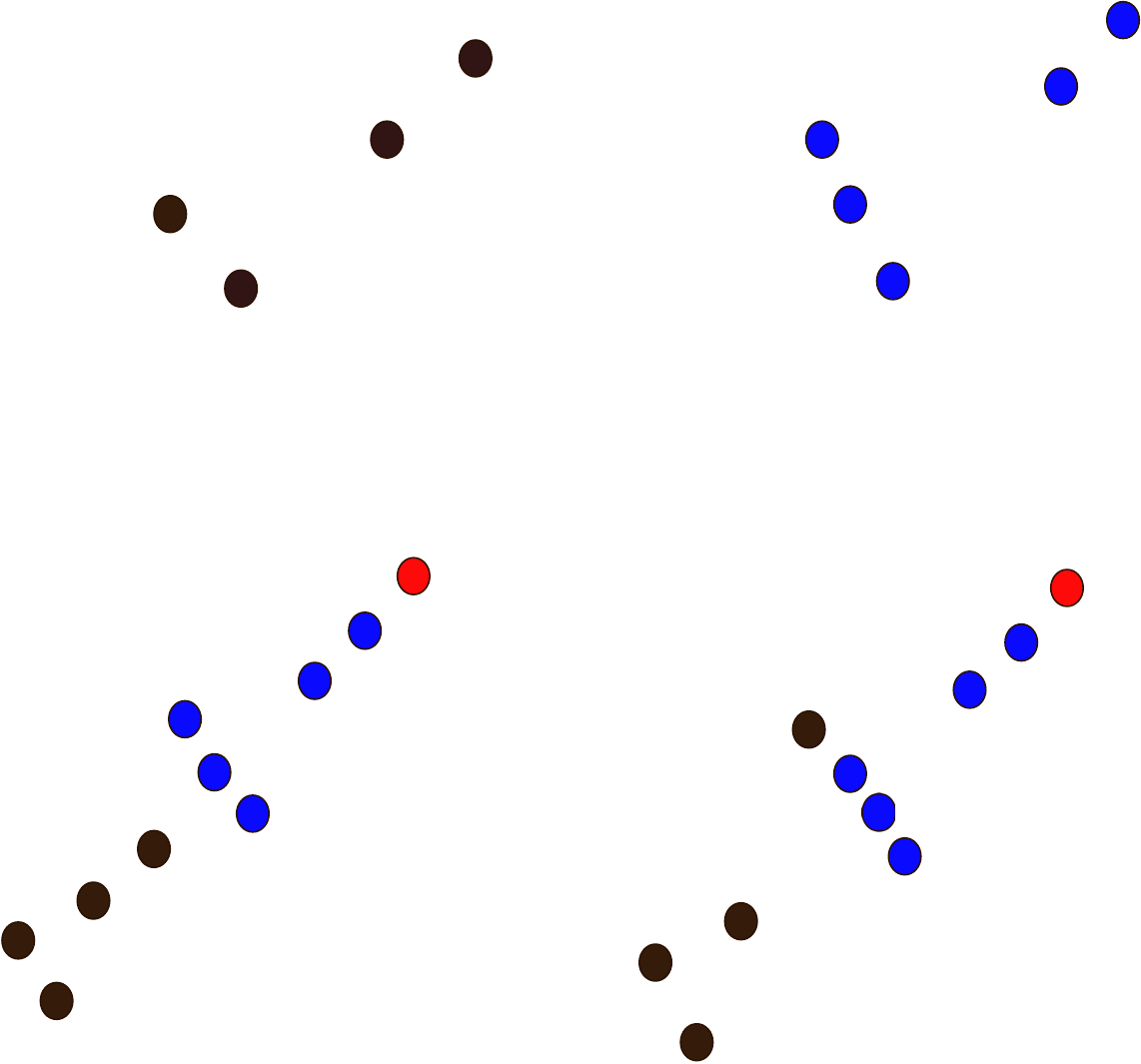}
\caption{The two ways of combining $s(L)$ and $s(R)$.}
\end{center} \end{figure}

\section{Further directions}
Now that we have found simple proofs for first two results mentioned in the Introduction, it is natural to ask for one for the third, that is, the equality $a_n=|s^{-1}(132,231)|$.
We point out that Defant \cite{enum} used the kernel method to prove that result, so it is possible that this case  is more complicated. If that is the case, that is interesting on its own, 
since that would mean that three permutation classes of identical size have stack sorting preimages of identical size, but for one of them, it is much more difficult to prove that than for the
other two.

\nocite{*}
\bibliographystyle{abbrvnat}
\bibliography{0-1-trees}
\label{sec:biblio}

\end{document}